\documentclass[10pt]{article}
\usepackage{amsfonts,amssymb,amsmath,amsthm,mathtools}
\usepackage{array,booktabs,longtable}
\usepackage{graphicx}
\usepackage{float}
\usepackage{tikz}
\usetikzlibrary{positioning}
\usepackage{caption}
\usepackage[a4paper,scale=0.8]{geometry}
\usepackage{enumitem}
\usepackage[numbers,sort&compress]{natbib}
\usepackage{hyperref}
\usepackage{cleveref}

\newtheorem{theorem}{Theorem}[section]

\newtheorem{lemma}{Lemma}[section]
\newtheorem{proposition}{Proposition}[section]

\theoremstyle{definition}
\newtheorem{definition}{Definition}[section]

\theoremstyle{remark}

\graphicspath{{Figures/}{output/pdf/Figures/}}
\makeatletter
\def\input@path{{./}{output/pdf/}}
\makeatother
\hypersetup{
    colorlinks=true,
    citecolor=cyan,
    linkcolor=cyan,
    filecolor=magenta,
    urlcolor=cyan,}

\newcommand{\be}{\begin{equation}}
\newcommand{\ee}{\end{equation}}

\crefname{figure}{fig.}{figures}
\crefname{table}{table}{tables}
\begin{document}
\title{The game chromatic number of generalized Mycielski graphs of paths and cycles}
\author{
    Yushuang Mou$^{\textnormal{a}}$, Qiang Sun$^{\textnormal{b}}$, Chao Zhang$^{\textnormal{a}}$%
    \\
    {\it \small $^{\textnormal{a}}$ State Key Laboratory of Public Big Data, School of Mathematics and Statistics,} \\
    {\it \small Guizhou University, Guiyang 550025, China} \\
    {\it \small $^{\textnormal{b}}$ School of Mathematical Science, Yangzhou University, Yangzhou 225009, China}
}

\date{}
\maketitle{}
\begingroup
\renewcommand{\thefootnote}{}
\footnotetext{\emph{E-mail addresses}: \protect\href{mailto:mys6875@163.com}{mys6875@163.com} (Y. Mou), \protect\href{mailto:qsun1987@163.com}{qsun1987@163.com} (Q. Sun), \protect\href{mailto:zhangc@amss.ac.cn}{zhangc@amss.ac.cn} (C. Zhang).}
\endgroup

\begin{abstract}
The graph coloring game is a two-player game in which the players alternately color an uncolored vertex of a graph $G$. The game chromatic number is the minimum number of colors needed for the first player to guarantee a win. We investigate this parameter for generalized Mycielski graphs $M_k(G)$, where $G$ is a path $P_n$ or a cycle $C_n$ with $n$ vertices. For every $k\geq2$ and $n\geq5$, we establish $4\leq\chi_g\bigl(M_k(P_n)\bigr)\leq5$ and
$4\leq\chi_g\bigl(M_k(C_n)\bigr)\leq5$. We also determine the exact values $\chi_g\bigl(M_2(P_5)\bigr)=\chi_g\bigl(M_2(P_6)\bigr)=4$. The proofs of the lower bounds use a configuration in which Bob can create two threats simultaneously, while the four-color upper bounds in the two exact cases are proved using the double-doctor lemma. Thus the number of layers and the order of the base graph may grow, but the game chromatic number remains bounded by five.
\end{abstract}

\noindent{\bf Keywords} graph coloring game; game chromatic number;
generalized Mycielski graphs

\vspace{0.5em}
\noindent{\bf Mathematics Subject Classification} 05C15, 05C57

%=========================================================
\section{Introduction}
%=========================================================

Given a simple graph $G$, the \emph{graph coloring game} is played by Alice and Bob, who alternately color an uncolored vertex of $G$ from a color set $X$. Each move must preserve a proper coloring; that is, adjacent colored vertices must receive different colors. The game begins with all vertices uncolored, and Alice moves first. Alice wins if every vertex is eventually colored; Bob wins if an uncolored vertex has no legal color.

The \emph{game chromatic number} $\chi_g(G)$, introduced by Bodlaender \cite{Bodlaender1991}, is the minimum number of colors for which Alice has a winning strategy. It extends the classical chromatic number and satisfies
\[
 \chi(G)\leq\chi_g(G)\leq\Delta(G)+1,
\]
where $\Delta(G)$ is the maximum degree of $G$. Whereas $\chi(G)$ records
whether a proper coloring exists, $\chi_g(G)$ reflects whether a proper coloring can be completed in the presence of adversarial choices. The difference between these two requirements can be substantial. Even among bipartite graphs, whose chromatic number is two, the game chromatic number is unbounded. If $M$ is a perfect matching of $K_{n,n}$, then $\chi_g(K_{n,n}-M)=n$~\cite{Bodlaender1991}. Moreover, deciding whether Alice has a winning strategy for a prescribed set of colors is \textsc{pspace}-complete~\cite{Costa2019}. Thus the game chromatic number records features of the evolving partial coloring that are not detected by the existence of a final proper coloring.

A central line of research seeks bounds that reflect the structure of the underlying graph. For example, $\chi_g(T)\leq4$ for every forest $T$ \cite{Faigle1993}, planar graphs have game chromatic number at most $17$ \cite{Zhu2008}, and outerplanar graphs have game chromatic number at
most $7$ \cite{GuanZhu1999}. Further developments have considered graph products, random graphs, and several other graph families \cite{Bartnicki2007,Enomoto2023,Hollom2025,Laso2017,Frieze2013,
Zhuproduct2008}. These results show that restrictions on the ordinary chromatic number, the maximum degree, or the presence of cycles do not by themselves determine the outcome of the game. A successful strategy must also control the order in which partial colorings arise and prevent several threats from becoming active at the same time.

The construction introduced by Mycielski \cite{Mycielski1955} increases the chromatic number without increasing the clique number. Starting with $K_2$ and applying the construction repeatedly produces graphs containing no triangles and having arbitrarily large chromatic number. Mycielski graphs therefore frequently serve as counterexamples to proposed bounds based only on clique size, the absence of triangles, or other local conditions. Beyond ordinary coloring, their fractional and circular chromatic numbers have been investigated in \cite{Larsen1995,Chang1999,Fan2004,Liu2004}.

Lam et al.~\cite{Lam2003} introduced generalized Mycielski graphs in their study of circular chromatic numbers. The generalized construction allows an arbitrary number of layers to be inserted before the root, and the case $k=1$ gives the usual Mycielski construction. The base graph is retained as an induced subgraph, while the number of layers can vary independently of the base graph. Consequently, this construction makes it possible to compare the behavior of a graph parameter as the number of layers varies while the structure of the base graph remains fixed.

Several structural and coloring parameters of generalized Mycielski graphs have subsequently been investigated. Lin et al.~\cite{Lin2006} examined their circular clique numbers, domination and packing
parameters, vertex cover numbers, spectra, and biclique partition numbers. Total coloring and backbone coloring were considered in \cite{Miskuf2009,Chen2014}, respectively. In contrast, comparatively little is known about their behavior in the graph coloring game. For $k=1$, the game chromatic numbers of Mycielski graphs have been studied for several base families, including complete graphs, complete
bipartite graphs, paths and cycles \cite{Chamberlin2019,Alagammai2019}. When additional layers are present, a vertex in an internal layer may have neighbors in both adjacent layers, while the root is adjacent only to vertices in the final layer. A move may therefore affect possible threats at several locations without changing the maximum degree of the non-root vertices. Hence the strategies developed for $k=1$ cannot be transferred directly to the generalized construction.

Paths and cycles provide a natural setting in which these effects can be distinguished. Both have maximum degree two, but a path has boundary vertices, whereas a cycle has a periodic structure. After the generalized Mycielski construction is applied, every non-root vertex has degree at most four, independently of the number of layers and the order of the base graph. The influence of the layers can therefore be examined without a corresponding increase in the local degree. Comparing the two families also reveals how boundary vertices and periodic structure affect the available coloring strategies.

A bound independent of both the number of layers and the order of the base graph separates the influence of the layered structure from that of increasing vertex degrees. Exact values further identify the local
configurations that permit Alice to improve such a bound. Treating $M_k(P_n)$ and $M_k(C_n)$ within the same framework makes it possible to determine whether these configurations depend on the boundary of a path or persist when the same local neighborhoods are arranged cyclically.

In this work, we investigate the game chromatic numbers of generalized Mycielski graphs obtained from paths and cycles. For every $k\geq2$, we determine the exact game chromatic numbers for the paths $P_2$, $P_3$, and $P_4$, and for the cycle $C_4$. We also establish
\[
 4\leq\chi_g(M_k(C_3))\leq5
\]
and, for $n\geq5$,
\[
 4\leq\chi_g(M_k(P_n))\leq5
 \qquad\text{and}\qquad
 4\leq\chi_g(M_k(C_n))\leq5.
\]
In particular, these bounds do not depend on the number of layers. We further prove
\[
 \chi_g(M_2(P_5))=\chi_g(M_2(P_6))=4.
\]
The lower bounds are obtained from local threat configurations, while coloring the root first gives a uniform five-color strategy. The two exact four-color results are proved by means of the double-doctor lemma. Thus the general bounds describe the behavior of the two graph families as their order and number of layers increase, while the exact cases identify additional local structure that permits the uniform upper bound to be reduced from five to four.

The remainder of the paper is organized as follows.
Section~\ref{sec:preliminaries} recalls the generalized Mycielski construction, fixes the notation used throughout the paper, and presents the definitions and known results needed later.
Section~\ref{sec:general-tools} establishes general bounds and structural lemmas that apply to both paths and cycles.
Section~\ref{sec:paths} studies generalized Mycielski graphs obtained from paths and includes the exact results for $M_2(P_5)$ and $M_2(P_6)$.
Section~\ref{sec:cycles} considers generalized Mycielski graphs obtained from cycles.

\allowdisplaybreaks

\section{Preliminaries}
\label{sec:preliminaries}

In this section, we recall the definition of the generalized Mycielski construction and fix the notation used throughout the paper. We then recall the known results for $k=1$ and state the definitions of partial colorings and safe vertices.
\subsection{Generalized Mycielski graphs and notation}
\label{subsec:mycielski-notation}
Let $G$ be a graph with vertex set $V^0= \{v_i^0: 1\leq i \leq n\}$ and edge set $E^0$. For any integer $k \geq 1$, the \emph{generalized Mycielski graph} $M_k(G)$ of $G$ is formed by stacking $k$ copies of $V(G)$, denoted $V^1, V^2, \dots, V^k$, and adding an extra vertex $\omega$. Specifically, the vertex set of $M_k(G)$ is

\begin{equation*}
V^0 \cup V^1 \cup V^2 \cup \dots \cup V^k \cup \{\omega\},
\end{equation*}
where each $V^m= \{v_i^m: v_i^0 \in V^0\}$ is the $m$-th copy of $V^0$, for $m=1,2,\dots,k$. The edge set of $M_k(G)$ is

\begin{equation*}
E^0 \cup \bigcup_{m=0}^{k-1}
\bigl\{v_i^m v_j^{m+1},v_j^m v_i^{m+1}:
v_i^0v_j^0\in E^0\bigr\}
\cup \{v_i^k\omega:v_i^k\in V^k\}.
\end{equation*}
Intuitively, $M_k(G)$ stacks $k$ copies of $V(G)$, connects each layer to the next, and adds a root $ \omega $ that is adjacent to every vertex in the final layer $V^k$ and to no other vertex. When $k=1$, $M_1(G)$ reduces to the classical \emph{Mycielski graph} $M(G)$. For example, the graphs $M(P_5)$ and $M(C_5)$ are illustrated in Fig.~\ref{fig1}, while the graphs $M_3(P_5)$ and $M_3(C_5)$ are shown in Fig.~\ref{fig2}.

\begin{figure}[htbp]
\centering
\begin{minipage}[b]{0.45\textwidth}
\centering
\includegraphics[scale=0.9]{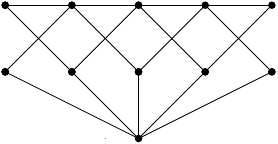}
 \end{minipage}
     \hspace{0.05em}
    \begin{minipage}[b]{0.45\textwidth}
\centering
\includegraphics[scale=0.9]{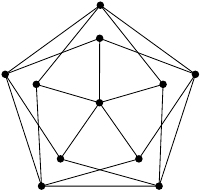}
 \end{minipage}
 \caption{The graphs $M(P_5)$ and $M(C_5)$.}
 \label{fig1}
\end{figure}

\begin{figure}[htbp]
\centering
\begin{minipage}[b]{0.45\textwidth}
\centering
\includegraphics[scale=0.9]{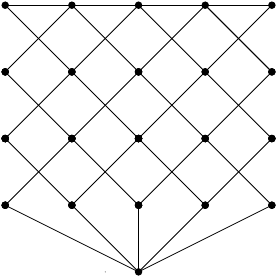}
 \end{minipage}
     \hspace{0.05em}
    \begin{minipage}[b]{0.45\textwidth}
\centering
\includegraphics[scale=0.9]{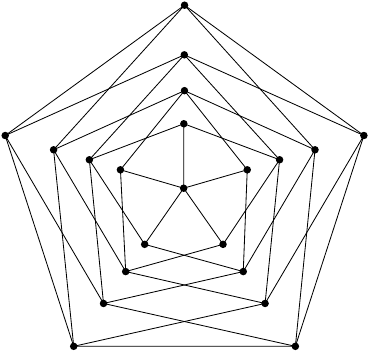}
 \end{minipage}
 \caption{The graphs $M_3(P_5)$ and $M_3(C_5)$.}
 \label{fig2}
\end{figure}

\subsection{Known results and preliminary definitions}
\label{subsec:game-background}
It is known that $\chi_g(P_n)=2$ for $2\leq n\leq3$ and $\chi_g(P_n)=3$ for $n\geq4$, and $\chi_g(C_n)=3$ for $n\geq 3$ \cite{Bodlaender1991,Alagammai2019}. We first recall the known exact values for the game chromatic numbers of the classical Mycielski graphs of paths and cycles.

\begin{proposition}[Alagammai and Vijayalakshmi~\cite{Alagammai2019}]
The game chromatic numbers of the Mycielski graphs of paths and cycles
are given as follows:
\begin{enumerate}[label={\upshape(\arabic*)},
                  nosep,
                  leftmargin=*]
\item For paths,
\[
\chi_g\bigl(M(P_n)\bigr)=
\begin{cases}
3, & 2\leq n\leq4,\\
4, & n\geq5.
\end{cases}
\]

\item For cycles,
\[
\chi_g\bigl(M(C_n)\bigr)=
\begin{cases}
3, & n=4,\\
4, & n=3\text{ or }n\geq5.
\end{cases}
\]
\end{enumerate}
\end{proposition}

For the remainder of the paper, $k\geq2$. We now consider $M_k(P_n)$ and $M_k(C_n)$. Let $\mathcal{M}_k$ denote the family consisting of these graphs, and let $X$ be the color set. For a vertex $v$, write $N(v)$ and $N[v]=N(v)\cup\{v\}$ for its open and closed neighborhoods, respectively, and write $\deg(v)=|N(v)|$. Every non-root vertex of a graph in $\mathcal{M}_k$ has degree at most four.

We next fix the notation used in the subsequent proofs.

\begin{definition}
A \emph{partial proper coloring} of a graph $G$ with color set $X$ is a map
\[
 \varphi:S\longrightarrow X,
 \qquad S\subseteq V(G),
\]
such that $\varphi(u)\ne\varphi(v)$ whenever $u,v\in S$ and $uv\in E(G)$.  The vertices in $S$ are colored and those in $V(G)\setminus S$ are uncolored. For an uncolored vertex $v$, put
\[
 C_\varphi(v)=\{\varphi(z):z\in N(v)\cap S\},
 \qquad L_\varphi(v)=X\setminus C_\varphi(v).
\]
The elements of $L_\varphi(v)$ are the \emph{legal colors} at $v$.
\end{definition}

The following notion of safety will be used throughout the paper.

\begin{definition}
\label{def:safe}
Let $G\in\mathcal{M}_k$, let $X$ be a color set with $|X|=4$, and let $\varphi$ be the current partial proper coloring. A vertex $v\in V(G)\setminus\{\omega\}$ is called \emph{safe} if at least one of the following conditions holds:
\begin{enumerate}[label={\upshape(\roman*)},leftmargin=*]
 \item $v$ is already colored;
 \item $\deg(v)\leq3$;
 \item two colored neighbors of $v$ have the same color under $\varphi$;
 \item exactly three neighbors of $v$ are colored with distinct colors,
 the remaining neighbor $u$ is uncolored, and $u$ has a colored neighbor $w$ such that $\varphi(w)$ is the fourth color.
\end{enumerate}
An uncolored vertex of degree four that satisfies none of these conditions is called \emph{unsafe}.
\end{definition}

Every safe vertex remains safe throughout the game. Indeed, this is immediate for an already colored vertex, and a vertex of degree at most three can never see all four colors. If two colored neighbors of $v$ have the same color, then at most three colors can occur in $N(v)$. Finally, suppose that condition~(iv) holds. The fourth color is unavailable at $u$ because it already occurs in $N(u)$. Hence, if $u$ is colored later, it must receive one of the three colors already present in $N(v)$, after which $v$ satisfies condition~(iii). Consequently, a safe vertex can never become uncolorable.

\section{General bounds and structural lemmas}
\label{sec:general-tools}

In this section, we establish several general lemmas that provide lower and upper bounds for generalized Mycielski graphs of paths and cycles. We first describe the configurations that yield lower bounds. We then prove a uniform upper bound of five and conclude with the double-doctor lemma, which is used to obtain the sharp upper bound of four in two path cases.

\subsection{Threat configurations and lemmas for lower bounds}
\label{subsec:threat-configurations}
For the three-color lower bounds, we use the following notion of a threat.

\begin{definition}
Let $t$ and $x$ be uncolored vertices with $x\in N(t)$.  The ordered pair $(t,x)$ is called an $a$-\emph{threat} if
\[
 C_\varphi(t)=X\setminus\{a\}
 \qquad\text{and}\qquad
 a\in L_\varphi(x).
\]
Here $t$ is the \emph{target}, $x$ is the \emph{attack vertex}, and $a$ is the \emph{threat color}. Thus $a$ is the unique legal color at $t$ and is also legal at $x$.
\end{definition}

If $(t,x)$ is an $a$-threat on Bob's turn, then Bob can color $x$ with $a$, after which $t$ has no legal color. In particular, the notation $(t_1,x_1)$ and $(t_2,x_2)$ below refers to two ordered pairs, not to graph edges considered in isolation.

\begin{lemma}[Double-threat lemma]
\label{lem:separated-double-threat}
Suppose that it is Alice's turn in a three-coloring game and that $(t_1,x_1)$ and $(t_2,x_2)$ are threats. Let $U$ be the set of currently uncolored vertices. If
\[
 \bigl(N[x_1]\cap U\bigr)\cap\bigl(N[x_2]\cap U\bigr)=\varnothing,
\]
then Bob has a winning strategy.
\end{lemma}

\begin{proof}
Let Alice color a vertex $z$. The two displayed response regions are disjoint, so $z$ lies in at most one of them. Hence, for some $i\in\{1,2\}$, Alice's move colors neither $t_i$ nor $x_i$ and does not assign the missing color of the threat to a neighbor of $x_i$. If Alice's move has already made an uncolored vertex uncolorable, Bob has won. Otherwise $(t_i,x_i)$ remains a threat, and Bob colors $x_i$ with its missing color. The vertex $t_i$ then has no legal color.
\end{proof}

The double-threat lemma applies when the two threats arise from a common cycle of length four. We describe this configuration in the following definition.
\begin{definition}
A \emph{Type~1 configuration} after Alice's first move consists of a $4$-cycle
\[
 u_1u_2u_3u_4u_1
\]
and two vertices $x_1,x_3$ such that $u_4$ is the only colored vertex, $u_1,u_2,u_3,x_1,x_3$ are uncolored,
\[
 x_1\in N(u_1)\setminus\{u_2,u_4\},\qquad
 x_3\in N(u_3)\setminus\{u_2,u_4\},
\]
and, after Bob colors $u_2$ with a color different from $\varphi(u_4)$,
\[
 \bigl(N[x_1]\cap U\bigr)\cap
 \bigl(N[x_3]\cap U\bigr)=\varnothing.
\]
Here $U$ denotes the set of vertices that remain uncolored after Bob's move.
\end{definition}

A Type~1 configuration is illustrated in Fig.~\ref{fig:type1-configuration}.

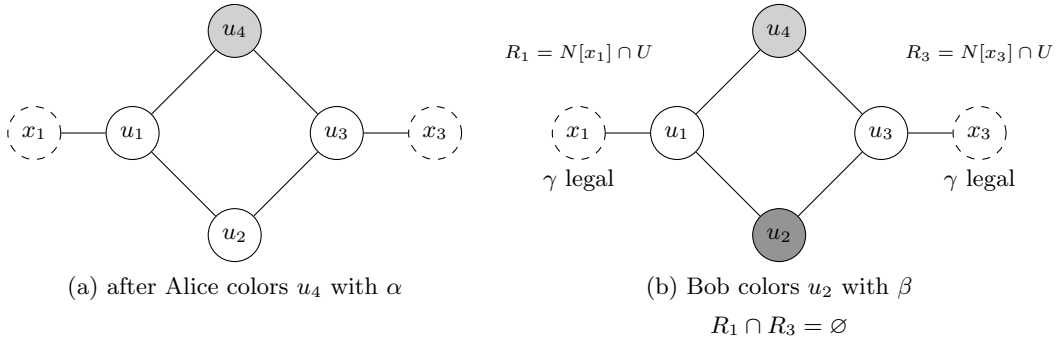
\begin{figure}[htbp]
\centering
\begin{tikzpicture}[
  x=1cm,y=1cm,
  vertex/.style={circle,draw,minimum size=7mm,inner sep=0pt,font=\small},
  alice/.style={vertex,fill=black!18},
  bob/.style={vertex,fill=black!42},
  attack/.style={vertex,dashed},
  every label/.style={font=\small}
]
% Position immediately after Alice's first move.
\begin{scope}[xshift=0cm]
 \node[alice]  (u4a) at (0,1.35) {$u_4$};
 \node[vertex] (u1a) at (-1.35,0) {$u_1$};
 \node[vertex] (u3a) at (1.35,0) {$u_3$};
 \node[vertex] (u2a) at (0,-1.35) {$u_2$};
 \node[attack] (x1a) at (-2.65,0) {$x_1$};
 \node[attack] (x3a) at (2.65,0) {$x_3$};
 \draw (u4a)--(u1a)--(u2a)--(u3a)--(u4a);
 \draw (x1a)--(u1a) (u3a)--(x3a);
 \node[font=\small] at (0,-2.05)
   {(a) after Alice colors $u_4$ with $\alpha$};
\end{scope}
% Position after Bob creates the two threats.
\begin{scope}[xshift=7.2cm]
 \node[alice]  (u4b) at (0,1.35) {$u_4$};
 \node[vertex] (u1b) at (-1.35,0) {$u_1$};
 \node[vertex] (u3b) at (1.35,0) {$u_3$};
 \node[bob]    (u2b) at (0,-1.35) {$u_2$};
 \node[attack,label=below:{$\gamma$ legal}] (x1b) at (-2.65,0) {$x_1$};
 \node[attack,label=below:{$\gamma$ legal}] (x3b) at (2.65,0) {$x_3$};
 \draw (u4b)--(u1b)--(u2b)--(u3b)--(u4b);
 \draw (x1b)--(u1b) (u3b)--(x3b);
 \node[font=\scriptsize,align=center] at (-2.65,1.05)
   {$R_1=N[x_1]\cap U$};
 \node[font=\scriptsize,align=center] at (2.65,1.05)
   {$R_3=N[x_3]\cap U$};
 \node[font=\small] at (0,-2.05)
   {(b) Bob colors $u_2$ with $\beta$};
 \node[font=\small] at (0,-2.55) {$R_1\cap R_3=\varnothing$};
\end{scope}
\end{tikzpicture}
\caption{A Type~1 configuration.  In panel~(b),
$(u_1,x_1)$ and $(u_3,x_3)$ are $\gamma$-threats.}
\label{fig:type1-configuration}
\end{figure}

\begin{lemma}\label{lemma1}
In a three-coloring game, if a Type~1 configuration occurs on Bob's turn immediately after Alice's first move, then Bob has a winning strategy.
\end{lemma}

\begin{proof}
Let $\varphi(u_4)=\alpha$, and let $\beta$ and $\gamma$ be the other two colors.  Bob colors $u_2$ with $\beta$. This move is legal because $u_4$ is the only previously colored vertex and $\beta\ne\alpha$. Now $u_1$ and $u_3$ both see the colors $\alpha$ and $\beta$, while $\gamma$ is legal at $x_1$ and $x_3$. Consequently, $(u_1,x_1)$ and $(u_3,x_3)$ are two $\gamma$-threats with disjoint response regions. Lemma~\ref{lem:separated-double-threat} completes the proof.
\end{proof}

The following lemma treats $M_k(C_6)$, an exceptional case in which the root is colored first. Throughout its statement and proof, all subscripts are interpreted modulo~$6$.

\begin{lemma}\label{lem:two-layer-fork}
Consider a coloring game on $M_k(C_6)$ with three colors, where $k\geq2$. Suppose that $v_i^0$ and $v_{i+2}^0$ have distinct colors $\alpha$ and $\beta$, and denote the third color by $\gamma$. Suppose that $v_{i+1}^0$ and $v_{i+1}^1$ are uncolored and that, for some $p,q\in\{i,i+2\}$, the vertices $v_p^1$ and $v_q^2$ are uncolored and $\gamma$ is legal at both of them. Then Bob has a winning strategy.
\end{lemma}

\begin{proof}
Both $v_{i+1}^0$ and $v_{i+1}^1$ are adjacent to $v_i^0$ and $v_{i+2}^0$.  Hence $(v_{i+1}^0,v_p^1)$ and
$(v_{i+1}^1,v_q^2)$ are $\gamma$-threats. Moreover,
\[
 N[v_p^1]\cap N[v_q^2]=\varnothing.
\]
Indeed, $v_p^1$ is adjacent only to vertices in $V^0$ and $V^2$ whose indices are $p-1$ or $p+1$, whereas $v_q^2$ is adjacent only to vertices in $V^1$ and $V^3$ whose indices are $q-1$ or $q+1$. When $k=2$, the vertex $v_q^2$ is adjacent to the root in place of the vertices in $V^3$. Since
\[
 p-q\equiv0,\pm2\pmod{6},
\]
these neighborhood descriptions imply that
\[
 N[v_p^1]\cap N[v_q^2]=\varnothing.
\]
The conclusion now follows from Lemma~\ref{lem:separated-double-threat}.
\end{proof}

\subsection{A uniform five-color upper bound for paths and cycles}
\label{subsec:five-color-bound}

We next establish a uniform upper bound for both graph families before considering the sharp upper bounds obtained with four colors. The proof shows that coloring the root on Alice's first move eliminates it as a possible obstruction and provides the degree estimate required later in the cycle case.
\begin{theorem}
\label{thm:uniform-five}
For every $k\geq2$, the following inequalities hold:
\[
  \chi_g\bigl(M_k(P_n)\bigr)\leq5\qquad(n\geq2)
\]
and
\[
  \chi_g\bigl(M_k(C_n)\bigr)\leq5\qquad(n\geq3).
\]

\end{theorem}

\begin{proof}
Let $G$ be either $P_n$ or $C_n$, with $n\geq2$ when $G=P_n$ and $n\geq3$ when $G=C_n$. Alice colors the root $\omega$ on her first move. Since the root has already been colored, it cannot become an obstruction. Every other vertex of $M_k(G)$ has degree at most four. Indeed, a vertex in $V^0$ has at most two neighbors in $V^0$ and at most two neighbors in $V^1$. For $1\leq m\leq k-1$, a vertex in $V^m$ has at most two neighbors in $V^{m-1}$ and at most two neighbors in $V^{m+1}$. Finally, a vertex in $V^k$ has at most two neighbors in $V^{k-1}$ and is also adjacent to the already colored root.

At every subsequent position, therefore, an uncolored vertex has at most four colors forbidden by its colored neighbors. Since five colors are available, every uncolored vertex has a legal color. The game cannot terminate before all vertices have been colored, and Alice wins. This proves both upper bounds.
\end{proof}

\subsection{The double-doctor lemma}
\label{subsec:double-doctor-method}

To give a unified treatment of the two sharp upper bounds for paths obtained with four colors, we first establish a general lemma based on assigning two safe neighbors to each unsafe vertex of degree four.

An unsafe vertex of degree four is called a \emph{patient}. Let $P$ be a set of patients. A \emph{double-doctor system} for $P$ is a family
\[
 \{D(x):x\in P\}
\]
such that, for every $x\in P$, the set $D(x)$ consists of two distinct uncolored safe neighbors of $x$, neither of which belongs to $P$, and the sets $D(x)$ are pairwise disjoint. The vertices in $D(x)$ are called the \emph{doctors} of $x$. Thus no vertex is assigned as a doctor to more than one patient, and no patient is a doctor. If $P$ is the set of all patients in the current position, then the double-doctor system is called a \emph{double-doctor certificate} for the position.

Thus the strategy and the resulting double-doctor certificate require four checks: Alice's response is legal; the patient list contains every unsafe vertex of degree four; each doctor is an uncolored safe neighbor of its patient and is not a patient; and the sets $D(x)$ are pairwise disjoint.

\begin{lemma}[Double-doctor lemma]
\label{lem:double-doctor}
Consider a position in the coloring game with four colors immediately after a move by Alice. If the root is colored and the position admits a double-doctor certificate, then Alice has a winning strategy from this position.
\end{lemma}

\begin{proof}
Keep the displayed doctor assignment until one of its vertices is used. After each move by Bob, first remove from the list any patient that Bob has colored. If no patient remains, Alice proceeds to the safe vertices. Otherwise, if Bob has colored a doctor of a remaining patient $x$, then the other doctor of $x$ is still uncolored. Hence at most three neighbors of $x$ are colored, so Alice colors $x$. If Bob has not colored a doctor of any remaining patient, Alice colors any remaining patient; both of its doctors are still uncolored, so this move is legal.

Because the sets $D(x)$ are pairwise disjoint and no patient is a doctor, these moves do not remove a doctor belonging to another patient. Thus every round strictly decreases the number of patients and preserves the double-doctor system for all remaining patients. After finitely many rounds, no patient remains. Since safety is permanent, no new patient can arise. Consequently, every remaining uncolored vertex is safe, and the coloring can be completed.
\end{proof}

\section{The coloring game on generalized Mycielski graphs of paths}
\label{sec:paths}

We first determine the exact values for the short paths and establish a general lower bound of four for all longer paths. Combined with the uniform upper bound in Theorem~\ref{thm:uniform-five}, this shows that the game chromatic numbers in the remaining cases are either four or five. We then apply the double-doctor lemma to show that \(M_2(P_5)\) and \(M_2(P_6)\) both have game chromatic number four.

\begin{theorem}
\label{thm:path-values}
For $k\geq2$, the following statements hold for generalized Mycielski graphs of paths:

\begin{enumerate}[label={\upshape(\arabic*)},
                  nosep,
                  leftmargin=*]
\item $\chi_g(M_k(P_2))=3$;
\item $\chi_g(M_k(P_3))=3$;
\item $\chi_g(M_k(P_4))=3$;
\item $\chi_g(M_k(P_n))\geq4$ for $n\geq5$.
\end{enumerate}
\end{theorem}

\begin{proof}
Let $V_i^\ast=\{v_i^m:0\leq m\leq k\}$, the set consisting of $v_i^0$ and all of its copies.
\vspace{0.3em}

\noindent (1) $\boldsymbol{\chi_g(M_k(P_2))=3.}$ \vspace{0.5em}

Since $M_k(P_2)$ is isomorphic to $C_{2(k+1)+1}$, it follows that $\chi_g(M_k(P_2))=\chi_g(C_{2k+3})=3$. \vspace{0.5em}

\noindent (2) $\boldsymbol{\chi_g(M_k(P_3))=3.}$

\begin{itemize}
\item \textit{Lower bound}: The subgraph induced by the root and the copies of the endpoints of any edge of $P_3$ is isomorphic to $M_k(P_2)\cong C_{2k+3}$. Thus $\chi(M_k(P_3))\geq3$, and consequently $\chi_g(M_k(P_3))\geq3$.

\item \textit{Upper bound}: To show that 3 colors suffice, assume $\lvert X \rvert = 3$. Alice colors $\omega$ on her first turn, then she executes the following strategy until all vertices in $V_2^\ast$ are colored.
\end{itemize}

\begin{enumerate}[label=(A\arabic*), labelindent=3.5em]
\item If Bob colors $v_1^m$ (or $v_3^m$) with color $\alpha$ for $m=0,1, \dots, k$, then Alice colors $v_3^m$ (or $v_1^m$) with the same color.

\item If Bob colors a vertex in $V_2^\ast$, then Alice colors any uncolored vertex in $V_2^\ast$ with a proper color.
\end{enumerate}

The two vertices paired in (A1) have identical open neighborhoods, so the copying move is always legal. At the beginning of each of Bob's turns, every such pair is either uncolored or colored with the same color. Therefore each uncolored vertex in $V_2^\ast$ sees at most two colors and has a legal color when Alice applies (A2). After all vertices in $V_2^\ast$ are colored, every remaining uncolored vertex has degree two, so Alice can complete the coloring with three colors. \vspace{0.5em}

\noindent (3) $\boldsymbol{\chi_g(M_k(P_4))=3.}$

\begin{itemize}
\item \textit{Lower bound}: As in (2), $M_k(P_4)$ contains an odd cycle isomorphic to $M_k(P_2)$, and hence $\chi_g(M_k(P_4))\geq3$.

\item \textit{Upper bound}: Let $\lvert X \rvert = 3$. Alice colors $\omega$ on her first turn, then proceeds as follows.
\end{itemize}

\begin{enumerate}[label=(A\arabic*), labelindent=3.5em]
\item If Bob colors $v_1^m$ (or $v_3^m$) with color $\alpha$ for $m=0,\dots,k$, then Alice colors $v_3^m$ (or $v_1^m$) with the same color.

\item If Bob colors $v_2^m$ (or $v_4^m$) with color $\alpha$ for $m=0,\dots,k$, then Alice colors $v_4^m$ (or $v_2^m$) with the same color.
\end{enumerate}

For each layer $V^m$, consider the pairs
$\{v_1^m,v_3^m\}$ and $\{v_2^m,v_4^m\}$. Alice maintains the following invariant: immediately before each move by Bob, the two vertices in every pair are either both uncolored or colored with the same color. We show that Alice can restore this invariant after each of Bob's moves.

In each pair, call the vertex of smaller degree the outer vertex and the other vertex the inner vertex. The neighborhood of the outer vertex is contained in that of the inner vertex. Hence, if Bob colors the inner vertex, his color is legal at the outer vertex. Conversely, suppose that Bob colors the outer vertex with a color that is not legal at the inner vertex. That color then occurs at a vertex in the difference of the two neighborhoods. The mate of this vertex belongs to the neighborhood of the outer vertex. By the invariant, the mate has the same color, which contradicts the legality of Bob's move. Thus, regardless of which vertex of a pair Bob colors, Alice can legally color the other vertex with the same color.

Finally, the neighborhood of every non-root vertex is contained in the union of at most two of these pairs. For a vertex in $V^k$, one of the pairs is replaced by the already colored root. At the beginning of each of Bob's turns, each pair contributes at most one color. After Bob's move, at most one pair is incomplete, and that pair still contributes only one color. Consequently, no uncolored vertex has neighbors of all three colors. Alice's prescribed response is therefore always legal, and all vertices are eventually colored.
 \vspace{0.5em}

\noindent (4) $\boldsymbol{\chi_g(M_k(P_n))\geq4\quad(n\geq5).}$

Assume that only three colors are available.  We show that Bob wins.

If Alice colors $\omega$, take
\[
 (u_1,u_2,u_3,u_4)
 =(v_2^k,v_3^{k-1},v_4^k,\omega),
 \qquad
 (x_1,x_3)=(v_1^{k-1},v_5^{k-1}).
\]
These vertices exist because $n\geq5$.  The four vertices $u_1,u_2,u_3,u_4$ form a $4$-cycle, and direct inspection of the neighborhoods gives $N[x_1]\cap N[x_3]=\varnothing$. Thus this is a Type~1 configuration.

Suppose next that Alice colors $v_i^m$.  Choose
\[
 s=\begin{cases}
 1,&i\leq n-2,\\
 -1,&i\in\{n-1,n\}.
 \end{cases}
\]
Table~\ref{tab:path-type1-configurations} specifies a Type~1 configuration. The entry in the second column is the ordered cycle $u_1u_2u_3u_4u_1$, with $u_4=v_i^m$ the vertex colored by Alice.

\begin{table}[htbp]
\centering
\caption{Type~1 configurations in $M_k(P_n)$ after Alice colors a non-root vertex.}
\label{tab:path-type1-configurations}
\footnotesize
\begin{tabular}{c|c|c|c}
\toprule
$m$ & $(u_1,u_2,u_3,u_4)$ & $x_1$ & $x_3$\\
\midrule
$0$
& $(v_{i+s}^0,v_{i+2s}^0,v_{i+s}^1,v_i^0)$
& $v_i^1$ & $v_i^2$\\
$1\leq m\leq k-2$
& $(v_{i+s}^{m-1},v_{i+2s}^m,v_{i+s}^{m+1},v_i^m)$
& $v_i^{\max\{0,m-2\}}$ & $v_i^{m+2}$\\
$k-1$
& $(v_{i+s}^{k-2},v_{i+2s}^{k-1},v_{i+s}^k,v_i^{k-1})$
& $v_i^{\max\{0,k-3\}}$ & $\omega$\\
$k$
& $(\omega,v_{i+2s}^k,v_{i+s}^{k-1},v_i^k)$
& $v_{i+s}^k$ & $v_i^{k-2}$\\
\bottomrule
\end{tabular}
\end{table}

The row $1\leq m\leq k-2$ in Table~\ref{tab:path-type1-configurations} is empty when $k=2$. In every remaining row, $x_1\in N(u_1)\setminus\{u_2,u_4\}$, $x_3\in N(u_3)\setminus\{u_2,u_4\}$, and the definition of $M_k(P_n)$ gives
\[
 N[x_1]\cap N[x_3]=\varnothing.
\]
Lemma~\ref{lemma1} therefore gives Bob a win in every case, and hence $\chi_g(M_k(P_n))\geq4$. \qedhere
\end{proof}

\begin{theorem}
\label{thm:M2P5-exact}
  $\chi_g\bigl(M_2(P_5)\bigr)=4$.

\end{theorem}

\begin{proof}
The lower bound is Theorem~\ref{thm:path-values}(4). We prove the upper bound. Use the notation
\[
 a_i=v_i^0,\qquad b_i=v_i^1,\qquad c_i=v_i^2
 \qquad(1\leq i\leq5),
\]
and write $\omega$ for the root. The neighborhoods are
\begin{equation}
\label{eq:M2P5-neighborhoods}
\begin{aligned}
N(a_i)&=\{a_{i-1},a_{i+1},b_{i-1},b_{i+1}\},\\
N(b_i)&=\{a_{i-1},a_{i+1},c_{i-1},c_{i+1}\},\\
N(c_i)&=\{b_{i-1},b_{i+1},\omega\},\\
N(\omega)&=\{c_1,\ldots,c_5\}.
\end{aligned}
\end{equation}
In these formulas, terms with indices outside $\{1,\ldots,5\}$ are omitted. Once $\omega$ is colored, the only vertices that can be patients are
\begin{equation}
\label{eq:M2P5-patients}
 D_5=\{a_2,a_3,a_4,b_2,b_3,b_4\}.
\end{equation}

%\begin{figure}[htbp]
%\centering
%\begin{tikzpicture}[
%  x=1.42cm,y=1.05cm,
%  every node/.style={circle,draw,inner sep=1.7pt,font=\small},
%  danger/.style={circle,draw,fill=black!18,inner sep=1.7pt,font=\small},
%  root/.style={circle,draw,fill=black!8,inner sep=2pt,font=\small}
%]
%\foreach \i in {1,5}{
%  \node (a\i) at (\i,2) {$a_{\i}$};
%  \node (b\i) at (\i,1) {$b_{\i}$};
%}
%\foreach \i in {2,3,4}{
%  \node[danger] (a\i) at (\i,2) {$a_{\i}$};
%  \node[danger] (b\i) at (\i,1) {$b_{\i}$};
%}
%\foreach \i in {1,...,5}{\node (c\i) at (\i,0) {$c_{\i}$};}
%\node[root] (w) at (3,-1.25) {$\omega$};
%\foreach \i [evaluate=\i as \j using int(\i+1)] in {1,...,4}{
%  \draw (a\i)--(a\j);
%  \draw (a\i)--(b\j) (a\j)--(b\i);
%  \draw (b\i)--(c\j) (b\j)--(c\i);
%}
%\foreach \i in {1,...,5}{\draw (c\i)--(w);}
%\end{tikzpicture}
%\caption{The graph $M_2(P_5)$. The vertices of degree four listed in
%\eqref{eq:M2P5-patients} are shaded.}
%\label{fig:M2P5-structure}
%\end{figure}

Alice first colors $\omega$ with a color $p$. Suppose that Bob then uses color $q$. By the reflection $i\mapsto6-i$, it is enough to consider Bob's move in columns $1,2,3$. Alice copies $q$ at the vertex specified in Table~\ref{tab:M2P5-first-response}.

\begin{table}[htbp]
\centering
\caption{Alice's first response in $M_2(P_5)$.}
\label{tab:M2P5-first-response}
\small
\begin{tabular}{c|ccccccccc}
\toprule
Bob & $a_1$&$b_1$&$c_1$&$a_2$&$b_2$&$c_2$&$a_3$&$b_3$&$c_3$\\
\midrule
Alice & $a_3$&$a_3$&$a_3$&$b_2$&$a_2$&$a_2$&$a_1$&$a_3$&$a_3$\\
\bottomrule
\end{tabular}
\end{table}

The two vertices in every column of the table are nonadjacent. Before Alice's response only $\omega$ and Bob's vertex have been colored, and none of the response vertices is adjacent to $\omega$. Hence every response is legal. Up to reflection, the two vertices receiving the same color form one of the seven pairs listed in Table~\ref{tab:M2P5-immediate}. In its first six rows, the table lists all patients and assigns two doctors to
each of them. The notation $x:(u,v)$ means that $u,v$ are the doctors of $x$.

\begin{table}[htbp]
\centering
\caption{Double-doctor certificates after Alice's first response in $M_2(P_5)$.}
\label{tab:M2P5-immediate}
\footnotesize
\begin{tabular}{c|l}
\toprule
vertices with the same color & patients and their doctors\\
\midrule
$a_1=a_3=q$
 & $a_4:(b_5,a_5)$; $b_3:(a_2,c_2)$; $b_4:(c_3,c_5)$\\
$b_1=a_3=q$
 & $a_4:(b_5,a_5)$; $b_2:(a_1,c_1)$;
   $b_3:(a_2,c_2)$; $b_4:(c_3,c_5)$\\
$c_1=a_3=q$
 & $a_2:(a_1,b_1)$; $a_4:(b_5,a_5)$;
   $b_3:(c_2,c_4)$; $b_4:(c_3,c_5)$\\
$a_3=b_3=q$
 & $b_2:(a_1,c_1)$; $b_4:(c_3,a_5)$\\
$a_3=c_3=q$
 & $a_2:(a_1,b_1)$; $a_4:(b_5,a_5)$; $b_3:(c_2,c_4)$\\
$a_2=b_2=q$
 & $a_4:(b_5,a_3)$; $b_3:(c_2,c_4)$; $b_4:(c_3,a_5)$\\
$a_2=c_2=q$ & deferred to Table~\ref{tab:M2P5-second-response}\\
\bottomrule
\end{tabular}
\end{table}

For clarity, we verify below that the patient column is complete. In each row, the two sets form a partition of the set $D_5$ of vertices of degree four defined in \eqref{eq:M2P5-patients}.

Table~\ref{tab:M2P5-patient-audit} verifies that the patient lists are complete.

\begin{table}[H]
\centering
\caption{Patient audit after Alice's first response in $M_2(P_5)$.}
\label{tab:M2P5-patient-audit}
\small
\begin{tabular}{c|l|l}
\toprule
vertices with the same color & colored or safe members of $D_5$ & patients\\
\midrule
$a_1=a_3$ & $a_2,a_3,b_2$ & $a_4,b_3,b_4$\\
$b_1=a_3$ & $a_2,a_3$ & $a_4,b_2,b_3,b_4$\\
$c_1=a_3$ & $a_3,b_2$ & $a_2,a_4,b_3,b_4$\\
$a_3=b_3$ & $a_2,a_3,a_4,b_3$ & $b_2,b_4$\\
$a_3=c_3$ & $a_3,b_2,b_4$ & $a_2,a_4,b_3$\\
$a_2=b_2$ & $a_2,a_3,b_2$ & $a_4,b_3,b_4$\\
$a_2=c_2$ & $a_2,b_3$ & $a_3,a_4,b_2,b_4$\\
\bottomrule
\end{tabular}
\end{table}

Indeed, the entries in the middle column are either already colored or have two neighbors with the same color. Every vertex in the last column is uncolored and satisfies none of the four conditions in Definition~\ref{def:safe}. Every non-root vertex outside $D_5$ has degree at most three, and the root is already colored. Thus the last column is exactly the patient set.

Every displayed doctor in Table~\ref{tab:M2P5-immediate} is uncolored and safe: it either has degree at most three, or is one of the vertices made safe by the indicated pair of vertices with the same color. The doctor lists are pairwise disjoint within each row. The doctors are adjacent to their patients by \eqref{eq:M2P5-neighborhoods}, so all conditions of a double-doctor certificate are satisfied. The double-doctor lemma therefore completes the first six rows.

It remains to treat $a_2=c_2=q$. Since $c_2$ is adjacent to $\omega$, we have $q\ne p$. Let Bob's second move be $z=r$. Table~\ref{tab:M2P5-second-response} gives Alice's response and, in the last column, the complete double-doctor certificate after that response. The conditions displayed in the middle column are the only possible ones: $q$ is unavailable at $a_1,b_1,a_3,b_3$, whereas $p$ is unavailable at every $c_i$.

\begingroup
\footnotesize
\renewcommand{\arraystretch}{1.08}
\begin{longtable}{c|c|c|p{0.43\textwidth}}
\caption{The second response in the exceptional position $a_2=c_2=q$.}
\label{tab:M2P5-second-response}\\
\toprule
$z$ & condition & Alice & patients and doctors\\
\midrule
\endfirsthead
\toprule
$z$ & condition & Alice & patients and doctors\\
\midrule
\endhead
\bottomrule
\endfoot
$a_1$ & $r\ne q$ & $a_3=r$
 & $a_4:(b_5,a_5)$; $b_4:(c_3,c_5)$\\
$b_1$ & $r\ne q$ & $a_3=r$
 & $a_4:(b_5,a_5)$; $b_2:(a_1,c_1)$; $b_4:(c_3,c_5)$\\
$a_3$ & $r\ne q$ & $b_3=r$
 & $b_2:(a_1,c_1)$; $b_4:(c_3,a_5)$\\
$b_3$ & $r\ne q$ & $a_3=r$
 & $b_2:(a_1,c_1)$; $b_4:(c_3,a_5)$\\
$a_4$ & any $r$ & $b_2=r$
 & $b_4:(a_3,c_3)$\\
$b_2$ & any $r$ & $a_4=r$
 & $b_4:(a_3,c_3)$\\
$b_4$ & any $r$ & $a_4=r$
 & $b_2:(a_1,c_1)$\\
$a_5$ & $r=q$ & $b_2=q$
 & $a_4:(b_5,a_3)$; $b_4:(c_3,c_5)$\\
$a_5$ & $r\ne q$ & $a_3=r$
 & $b_2:(a_1,c_1)$\\
$b_5$ & $r=q$ & $b_2=q$
 & $a_4:(a_3,a_5)$; $b_4:(c_3,c_5)$\\
$b_5$ & $r\ne q$ & $a_3=r$
 & $b_2:(a_1,c_1)$; $b_4:(c_3,a_5)$\\
$c_1$ & $r=q$ & $a_4=q$
 & $b_2:(a_1,a_3)$; $b_4:(c_3,a_5)$\\
$c_1$ & $r\ne q$ & $a_3=r$
 & $a_4:(b_5,a_5)$; $b_4:(c_3,c_5)$\\
$c_3$ & $r=q$ & $a_4=q$
 & $b_2:(a_1,c_1)$; $b_4:(a_3,a_5)$\\
$c_3$ & $r\ne q$ & $a_3=r$
 & $a_4:(b_5,a_5)$\\
$c_4$ & $r=q$ & $a_4=q$
 & $b_2:(a_1,c_1)$; $b_4:(a_3,c_3)$\\
$c_4$ & $r\ne q$ & $a_3=r$
 & $a_4:(b_5,a_5)$; $b_2:(a_1,c_1)$; $b_4:(c_3,c_5)$\\
$c_5$ & $r=q$ & $a_4=q$
 & $b_2:(a_1,c_1)$; $b_4:(a_3,c_3)$\\
$c_5$ & $r\ne q$ & $a_3=r$
 & $a_4:(b_5,a_5)$; $b_2:(a_1,c_1)$\\
\end{longtable}
\endgroup

Table~\ref{tab:M2P5-second-response} is exhaustive because $\omega,a_2,c_2$ are already colored. Every response of color $r$ copies Bob at a nonadjacent vertex. If $r=q$, the response is instead at $a_4$ or $b_2$, neither of which is adjacent to $a_2$ or $c_2$; if $r=p$, no response vertex is adjacent to
$\omega$. Thus all responses are legal.

It remains only to verify the certificates. In each row, the patients shown in the last column are exactly those vertices of $D_5$ that are neither colored nor made safe by the displayed pair of vertices with the same color; this verifies patient completeness directly from \eqref{eq:M2P5-neighborhoods}. Boundary vertices $a_1,a_5,b_1,b_5$ and all vertices $c_i$ have degree at most three. Any listed internal doctor $a_3$ has two already colored neighbors of the same color: these are either $a_2,b_2$, $a_2,a_4$, or $a_4,b_2$, as the corresponding row shows. Hence every listed doctor is safe. The lists in each row are disjoint, contain no colored vertex, and each
doctor lies in the neighborhood of its patient by \eqref{eq:M2P5-neighborhoods}. Therefore every row satisfies Lemma~\ref{lem:double-doctor}. Alice wins with four colors, proving $\chi_g(M_2(P_5))\leq4$. Together with the lower bound, this gives the claimed equality.
\end{proof}

\begin{theorem}
\label{thm:M2P6-exact}
$\chi_g(M_2(P_6))=4$.
\end{theorem}

\begin{proof}
The lower bound follows from Theorem~\ref{thm:path-values}(4). We prove the upper bound by giving Alice a four-color strategy.

Write
\[
 a_i=v_i^0,\qquad b_i=v_i^1,\qquad c_i=v_i^2
 \qquad(1\leq i\leq6),
\]
and denote the root by $\omega$. The neighborhoods are
\begin{equation}
\label{eq:M2P6-neighborhoods}
\begin{aligned}
 N(a_i)&=\{a_{i-1},a_{i+1},b_{i-1},b_{i+1}\},\\
 N(b_i)&=\{a_{i-1},a_{i+1},c_{i-1},c_{i+1}\},\\
 N(c_i)&=\{b_{i-1},b_{i+1},\omega\},\\
N(\omega)&=\{c_1,\ldots,c_6\}.
\end{aligned}
\end{equation}
In these formulas, terms with indices outside $\{1,\ldots,6\}$ are omitted. After $\omega$ has been colored, the only vertices that can become patients are
\begin{equation}
\label{eq:M2P6-patients}
 D_6=\{a_2,a_3,a_4,a_5,b_2,b_3,b_4,b_5\}.
\end{equation}

Alice first colors $\omega$ with color $p$. By the reflection $i\mapsto7-i$, it is enough to consider Bob's first move in columns $1$, $2$, and $3$. If Bob uses color $q$, Alice assigns the same color to the vertex specified in Table~\ref{tab:M2P6-first-response}.

\begin{table}[htbp]
\centering
\caption{Alice's first response in $M_2(P_6)$.}
\label{tab:M2P6-first-response}
\small
\begin{tabular}{c|ccccccccc}
\toprule
Bob & $a_1$&$a_2$&$a_3$&$b_1$&$b_2$&$b_3$&$c_1$&$c_2$&$c_3$\\
\midrule
Alice & $a_3$&$a_4$&$a_5$&$a_3$&$a_2$&$a_3$&$a_3$&$a_4$&$a_3$\\
\bottomrule
\end{tabular}
\end{table}

Every pair in the table is nonadjacent. Since only $\omega$ and Bob's vertex were colored before Alice's response, every listed response is legal. Three types already have a double-doctor certificate; the complete certificates are displayed in Table~\ref{tab:M2P6-immediate-certificates}.

\begin{table}[htbp]
\centering
\caption{Double-doctor certificates after Alice's first response in $M_2(P_6)$.}
\label{tab:M2P6-immediate-certificates}
\footnotesize
\begin{tabular}{c|c|p{0.58\textwidth}}
\toprule
vertices with the same color & patients & doctors\\
\midrule
$a_2=a_4=q$
& $a_5,b_2,b_4,b_5$
& $a_5:(a_6,b_6)$; $b_2:(a_1,c_1)$;
  $b_4:(a_3,c_3)$; $b_5:(c_4,c_6)$\\
\addlinespace
$a_3=a_5=q$
& $a_2,b_2,b_3,b_5$
& $a_2:(a_1,b_1)$; $b_2:(c_1,c_3)$;
  $b_3:(c_2,a_4)$; $b_5:(c_4,a_6)$\\
\addlinespace
$a_3=b_3=q$
& $a_5,b_2,b_4,b_5$
& $a_5:(a_4,a_6)$; $b_2:(a_1,c_1)$;
  $b_4:(c_3,c_5)$; $b_5:(c_4,c_6)$\\
\bottomrule
\end{tabular}
\end{table}

Indeed, the following equalities ensure that the indicated middle vertices are safe:
\[
\begin{aligned}
 a_2=a_4=q&\Longrightarrow a_3,b_3\text{ are safe},\\
 a_3=a_5=q&\Longrightarrow a_4,b_4\text{ are safe},\\
 a_3=b_3=q&\Longrightarrow a_2,a_4\text{ are safe}.
\end{aligned}
\]
Together with the two colored members of each pair, these safe vertices form the complement of the patient set displayed in Table~\ref{tab:M2P6-immediate-certificates} within the set $D_6$ defined in \eqref{eq:M2P6-patients}. Hence that table lists every patient. Each doctor pair lies in the appropriate neighborhood by \eqref{eq:M2P6-neighborhoods}; its vertices are uncolored and safe, and the pairs in each row are disjoint. Thus the table supplies complete double-doctor certificates, not only candidate doctor lists. Lemma~\ref{lem:double-doctor} now finishes these three types.

It remains to consider Bob's first move at $a_1,b_1,b_2,c_1,c_2$, or $c_3$.  Let $R_x$ denote the position after Bob first colors $x$ with $q$ and Alice makes the response from Table~\ref{tab:M2P6-first-response}. Suppose that Bob next colors $z$ with $r$. Alice's second response is given in Table~\ref{tab:M2P6-second-response}. In the table, $u(r)$ and $u(q)$ indicate that Alice colors $u$ with $r$ and $q$, respectively.

\begingroup
\footnotesize
\renewcommand{\arraystretch}{1.08}
\begin{longtable}{c|p{0.28\textwidth}|p{0.43\textwidth}}
\caption{Alice's second response in the six remaining types of $M_2(P_6)$.}
\label{tab:M2P6-second-response}\\
\toprule
position & Bob's second vertex $z$ & Alice's response\\
\midrule
\endfirsthead
\toprule
position & Bob's second vertex $z$ & Alice's response\\

\endhead
\bottomrule
\endfoot
$R_{a_1}$
& $a_2,b_2,b_4$ & $a_4(r)$\\
& $a_4$ & $a_2(r)$\\
& $a_5$ & $b_3(r)$\\
& $a_6,b_6$ & $b_3(q)$ if $r=q$; $a_4(r)$ otherwise\\
& $b_1$ & $a_5(q)$ if $r=q$; $a_4(r)$ otherwise\\
& $b_3,b_5$ & $a_5(r)$\\
& $c_1,c_2,c_4,c_6$ & $a_5(q)$ if $r=q$; $a_4(r)$ otherwise\\
& $c_3,c_5$ & $a_5(q)$\\

$R_{b_1}$
& $a_1$ & $a_5(q)$\\
& $a_2,b_2,b_4$ & $a_4(r)$\\
& $a_4$ & $a_2(r)$\\
& $a_5$ & $b_3(r)$\\
& $a_6,b_6$ & $b_3(q)$ if $r=q$; $a_4(r)$ otherwise\\
& $b_3,b_5$ & $a_5(r)$\\
& $c_1,c_3,c_5$ & $a_5(q)$\\
& $c_2$ & $a_4(r)$\\
& $c_4,c_6$ & $a_5(q)$ if $r=q$; $a_4(r)$ otherwise\\
\midrule
$R_{b_2}$
& $a_1,b_1$ & $a_4(q)$\\
& $a_3$ & $a_5(r)$\\
& $a_4$ & $b_4(r)$\\
& $a_5$ & $b_5(q)$ if $r=q$; $a_3(r)$ otherwise\\
& $a_6,b_4,b_6$ & $a_4(r)$\\
& $b_3$ & $a_3(r)$\\
& $b_5$ & $a_5(q)$ if $r=q$; $a_3(r)$ otherwise\\
& $c_1,c_3$ & $a_4(q)$\\
& $c_2,c_4,c_6$ & $a_4(r)$\\
& $c_5$ & $a_4(q)$\\
\midrule
$R_{c_1}$
& $a_1,b_1$ & $b_3(q)$\\
& $a_2,b_4$ & $a_4(r)$\\
& $a_4$ & $a_2(r)$\\
& $a_5$ & $b_3(r)$\\
& $a_6,b_6$ & $b_3(q)$ if $r=q$; $a_4(r)$ otherwise\\
& $b_2$ & $a_5(q)$\\
& $b_3,b_5$ & $a_5(r)$\\
& $c_2,c_4,c_6$ & $a_5(q)$ if $r=q$; $a_4(r)$ otherwise\\
& $c_3,c_5$ & $a_5(q)$\\
\midrule
$R_{c_2}$
& $a_1$ & $b_4(q)$ if $r=q$; $a_3(r)$ otherwise\\
& $a_2$ & $b_2(r)$\\
& $a_3$ & $a_5(r)$\\
& $a_5,b_1,b_3,b_5$ & $a_3(r)$\\
& $a_6,b_6$ & $a_2(q)$ if $r=q$; $b_4(q)$ otherwise\\
& $b_2,b_4$ & $a_2(r)$\\
& $c_1,c_3,c_5$ & $a_2(q)$ if $r=q$; $a_3(r)$ otherwise\\
& $c_4,c_6$ & $a_2(q)$ if $r=q$; $b_4(q)$ otherwise\\
\midrule
$R_{c_3}$
& $a_1,b_1$ & $b_3(q)$\\
& $a_2,b_2,b_4$ & $a_4(r)$\\
& $a_4$ & $a_2(r)$\\
& $a_5$ & $b_3(r)$\\
& $a_6,b_6$ & $b_3(q)$ if $r=q$; $a_4(r)$ otherwise\\
& $b_3,b_5$ & $a_5(r)$\\
& $c_1,c_2,c_4,c_5,c_6$ & $a_5(q)$ if $r=q$; $a_4(r)$ otherwise\\
\end{longtable}
\endgroup

For each position $R_x$, the sets listed in the second column of Table~\ref{tab:M2P6-second-response} form a partition of the vertices that remain uncolored after Alice's first response. Hence the table is exhaustive. The legality of every response follows directly from \eqref{eq:M2P6-neighborhoods}. For an entry $u(r)$, the vertex $u$ is uncolored and is nonadjacent to $z$ and to every vertex that was previously assigned color $r$. For an entry $u(q)$, the vertex $u$ is uncolored and is nonadjacent to every vertex that currently has color $q$, including $z$ when $r=q$. Thus every response prescribed in the table is legal.

Table~\ref{tab:M2P6-doctor-order} gives the candidate order used to construct a double-doctor certificate. After each row of Table~\ref{tab:M2P6-second-response}, process the remaining patients in the fixed order
\[
 b_3,b_4,a_5,a_4,a_2,b_2,b_5,a_3,
\]
and assign to each patient the first two vertices in its list that are uncolored, safe, and not already used as doctors.

\begin{table}[H]
\centering
\caption{Ordered candidate lists for assigning doctors in $M_2(P_6)$.}
\label{tab:M2P6-doctor-order}
\small
\begin{tabular}{c|l@{\qquad}c|l}
\toprule
patient & candidate order & patient & candidate order\\
\midrule
$b_3$ & $c_2,a_2,a_4,c_4$ & $b_4$ & $a_5,c_3,c_5,a_3$\\
$a_5$ & $b_4,a_4,b_6,a_6$ & $a_4$ & $a_5,b_5,b_3,a_3$\\
$a_2$ & $b_3,a_3,b_1,a_1$ & $b_2$ & $c_1,a_1,a_3,c_3$\\
$b_5$ & $c_6,a_6,a_4,c_4$ & $a_3$ & $a_4,b_4,b_2,a_2$\\
\bottomrule
\end{tabular}
\end{table}

Each list in Table~\ref{tab:M2P6-doctor-order} is the neighborhood of its patient, written in the prescribed order. Table~\ref{tab:M2P6-greedy-audit} verifies the selection procedure defined by the fixed order of the patients and the ordered candidate lists. An entry in this table is the minimum number of eligible candidates remaining when the corresponding patient is reached, taken over every legal second move in the indicated position and over all legal relations among $p$, $q$, and $r$. A dash indicates that the corresponding vertex is never a patient after Alice's second response.

\begin{table}[H]
\centering
\caption{Verification matrix for the doctor selection procedure.}
\label{tab:M2P6-greedy-audit}
\small
\begin{tabular}{c|cccccccc}
\toprule
position & $b_3$ & $b_4$ & $a_5$ & $a_4$ & $a_2$ & $b_2$ & $b_5$ & $a_3$\\
\midrule
$R_{a_1}$ & 3 & 2 & 2 & -- & -- & -- & 2 & --\\
$R_{b_1}$ & 3 & 2 & 2 & -- & -- & 2 & 2 & --\\
$R_{b_2}$ & 2 & 2 & 2 & -- & -- & -- & 2 & --\\
$R_{c_1}$ & 2 & 2 & 2 & -- & 2 & -- & 2 & --\\
$R_{c_2}$ & -- & 2 & 2 & -- & 3 & 2 & 2 & --\\
$R_{c_3}$ & 2 & -- & 3 & -- & 2 & -- & 2 & --\\
\bottomrule
\end{tabular}
\end{table}

We now describe the verification procedure used to obtain Table~\ref{tab:M2P6-greedy-audit}. For each case listed in Table~\ref{tab:M2P6-second-response}, let $C$ be the set of the five vertices that have already been colored: the root, the vertices colored on Bob's two moves, and the vertices colored on Alice's two response moves. After Alice's second response, define
\[
 P=\{x\in D_6:x\text{ is unsafe}\}.
\]
By \eqref{eq:M2P6-patients} and Definition~\ref{def:safe}, this is exactly the set of patients in the current position.

Process the vertices of $P$ in the fixed order displayed above. For the current patient $x$, examine the vertices in its candidate list in Table~\ref{tab:M2P6-doctor-order} in the prescribed order, excluding
every vertex that belongs to $C$, is unsafe, or has already been assigned as a doctor to an earlier patient. Since the list in Table~\ref{tab:M2P6-doctor-order} is the neighborhood of $x$, every remaining vertex is an uncolored safe neighbor of $x$ that has not previously been used as a doctor.

For each initial position, the corresponding entry in Table~\ref{tab:M2P6-greedy-audit} is the minimum number of eligible candidates remaining when the indicated patient is processed, where the minimum is taken over all legal choices of Bob's second move and all admissible relations among the colors $p$, $q$, and $r$. Applying this procedure to every case in Table~\ref{tab:M2P6-second-response} yields the six rows of the verification matrix.

As an illustration of this procedure, consider the position $R_{a_1}$ in which Bob colors $c_3$ and Alice responds by coloring $a_5$ with color $q$. Since $a_3=a_5=q$, the vertices $a_4$ and $b_4$ are safe. Hence the exact patient set is $\{b_3,b_5\}$. Following the prescribed order, the procedure assigns
\[
 D(b_3)=\{a_2,c_2\}
 \qquad\text{and}\qquad
 D(b_5)=\{a_6,c_6\}.
\]

Every numerical entry in Table~\ref{tab:M2P6-greedy-audit} is at least two. Therefore, whenever a patient is processed, at least two eligible candidates remain, and the procedure can assign two doctors to that patient. Since every vertex already assigned as a doctor is excluded from all subsequent assignments, the sets $D(x)$ are pairwise disjoint. Moreover, every assigned doctor is an uncolored safe neighbor of its patient. Since every patient is unsafe whereas every doctor is safe, no patient is a doctor. Thus the resulting assignment is a double-doctor certificate.

Consequently, Alice establishes a double-doctor certificate no later than her third move. Lemma~\ref{lem:double-doctor} gives
\[
 \chi_g(M_2(P_6))\leq4.
\]
Together with the lower bound, this proves
\[
 \chi_g(M_2(P_6))=4.
\]
\end{proof}

Theorems~\ref{thm:M2P5-exact} and~\ref{thm:M2P6-exact} show that the general lower bound in Theorem~\ref{thm:path-values}(4) is attained when $k=2$ and $n\in\{5,6\}$. In both proofs, the boundary vertices of the path are used in assigning doctors. We now turn to cycles. The uniform strategy with five colors remains valid, whereas the configurations used to establish the lower bound must be adapted to the cyclic structure.

\section{The coloring game on generalized Mycielski graphs of cycles}
\label{sec:cycles}

In this section, we apply the structural lemmas from Section~\ref{sec:general-tools} to establish bounds for generalized Mycielski graphs of cycles.

\begin{theorem}
\label{thm:cycle-values}
For $k\geq2$, the following statements hold for generalized Mycielski graphs of cycles:

\begin{enumerate}[label={\upshape(\arabic*)},
                  nosep,
                  leftmargin=*]
\item $4\leq\chi_g(M_k(C_3))\leq5$;
\item $\chi_g(M_k(C_4))=3$;
\item $4\leq\chi_g(M_k(C_n))\leq5$ for $n\geq5$.
\end{enumerate}
\end{theorem}

\begin{proof}
\noindent (1) $\boldsymbol{4 \leq \chi_g(M_k(C_3)) \leq 5.}$ \vspace{0.5em}

The generalized Mycielski construction gives $\chi(M_k(C_3))=4$ \cite{Lam2003}, and $\Delta(M_k(C_3))=4$. Therefore $\chi(G)\leq\chi_g(G)\leq\Delta(G)+1$ yields $4\leq\chi_g(M_k(C_3))\leq5$.
 \vspace{0.5em}

\noindent (2) $\boldsymbol{\chi_g(M_k(C_4))=3.}$ \vspace{0.5em}

\begin{itemize}
\item \textit{Lower bound}: The graph contains $M_k(P_2)\cong C_{2k+3}$ as a subgraph, so it is not bipartite. Hence $\chi_g(M_k(C_4))\geq\chi(M_k(C_4))\geq3$.

\item \textit{Upper bound}: Let $\lvert X\rvert=3$. Alice first colors $\omega$ with color~$1$ and then uses the following strategy.
\end{itemize}

\begin{enumerate}[label=(A\arabic*), labelindent=3.5em]
\item If Bob colors $v_1^m$ (or $v_3^m)$ for $m=0,\dots,k$ with color $\alpha$, then Alice colors $v_3^m$ (or $v_1^m)$ with the same color.

\item If Bob colors $v_2^m$ (or $v_4^m)$ for $m=0,\dots,k$ with color $\alpha$, then Alice colors $v_4^m$ (or $v_2^m)$ with the same color.
\end{enumerate}

Opposite vertices in a fixed layer have identical open neighborhoods, so every copying move is legal. After Alice's moves, every opposite pair is either uncolored or colored with the same color. The neighborhood of each non-root vertex is the union of at most two such pairs, with the already colored root replacing one pair at the top layer. Immediately after a Bob move, at most one pair is incomplete. Hence an uncolored vertex sees at most two colors and can never be blocked when three colors are available. Thus Alice wins. \vspace{0.5em}

\noindent (3) $\boldsymbol{4\leq\chi_g(M_k(C_n))\leq5
\quad(n\geq5).}$

For the lower bound, suppose that only three colors are available. All subscripts in what follows are taken modulo $n$.

First suppose that Alice colors a non-root vertex $v_i^m$. Table~\ref{tab:cycle-type1-configurations} gives a Type~1 configuration; the second column is the ordered cycle $u_1u_2u_3u_4u_1$, where $u_4=v_i^m$ is Alice's vertex.

\begin{table}[htbp]
\centering
\caption{Type~1 configurations in $M_k(C_n)$ after Alice colors a non-root vertex.}
\label{tab:cycle-type1-configurations}
\footnotesize
\begin{tabular}{c|c|c|c}
\toprule
$m$ & $(u_1,u_2,u_3,u_4)$ & $x_1$ & $x_3$\\
\midrule
$0$
& $(v_{i+1}^0,v_{i+2}^0,v_{i+1}^1,v_i^0)$
& $v_i^1$ & $v_i^2$\\
$1\leq m\leq k-2$
& $(v_{i+1}^{m-1},v_{i+2}^m,v_{i+1}^{m+1},v_i^m)$
& $v_i^{\max\{0,m-2\}}$ & $v_i^{m+2}$\\
$k-1$
& $(v_{i+1}^{k-2},v_{i+2}^{k-1},v_{i+1}^k,v_i^{k-1})$
& $v_i^{\max\{0,k-3\}}$ & $\omega$\\
$k$
& $(\omega,v_{i+2}^k,v_{i+1}^{k-1},v_i^k)$
& $v_{i+1}^k$ & $v_i^{k-2}$\\
\bottomrule
\end{tabular}
\end{table}

As before, the second data row of Table~\ref{tab:cycle-type1-configurations} is empty when $k=2$. Since $n\geq5$, the definition of $M_k(C_n)$ gives
\[
 N[x_1]\cap N[x_3]=\varnothing
\]
in every row. Lemma~\ref{lemma1} therefore applies.

Next suppose that Alice colors $\omega$. If $n=5$ or $n\geq7$, take
\[
 (u_1,u_2,u_3,u_4)
 =(v_i^k,v_{i+1}^{k-1},v_{i+2}^k,\omega),
 \qquad
 (x_1,x_3)=(v_{i-1}^{k-1},v_{i+3}^{k-1}).
\]
The two closed neighborhoods are disjoint: an intersection would force one of the congruences $4\equiv0$ or $6\equiv0\pmod n$, neither of which is possible for $n=5$ or $n\geq7$. Hence Lemma~\ref{lemma1} again gives Bob a win.

It remains to treat $n=6$ when Alice colors the root. Rename the colors so that $\varphi(\omega)=0$. Bob colors $v_0^0$ with color $0$, which is legal because the root is not adjacent to the base layer. Let Alice's next move be at $z$. Colors $1$ and $2$ are still symmetric, so if $\varphi(z)\ne0$ we rename them so that $\varphi(z)=1$. Up to the reflection $i\mapsto-i\pmod6$, Bob's response is specified in Table~\ref{tab:C6-root-response}. An entry $(\varepsilon,\beta,p,q)$ means that Bob colors $v_{2\varepsilon}^0$ with $\beta$ and uses $v_p^1,v_q^2$ as the two attack vertices.

\begin{table}[htbp]
\centering
\caption{Bob's responses after Alice colors the root in $M_k(C_6)$.}
\label{tab:C6-root-response}
\footnotesize
\begin{tabular}{c|c|c}
\toprule
$z$ & $\varphi(z)$ & $(\varepsilon,\beta,p,q)$\\
\midrule
$v_1^0$ & $1$ & $(-1,1,4,0)$\\
$v_2^0$ & $0\text{ or }1$ & $(-1,1,4,0)$\\
$v_3^0$ & $0$ & $(1,1,2,0)$\\
$v_3^0$ & $1$ & $(1,2,0,0)$\\
$v_0^1$ & $0\text{ or }1$ & $(1,1,2,0)$\\
$v_1^1$ & $1$ & $(-1,1,4,0)$\\
$v_2^1$ & $0\text{ or }1$ & $(1,1,0,0)$\\
$v_3^1$ & $0$ & $(1,1,2,0)$\\
$v_3^1$ & $1$ & $(1,2,2,0)$\\
$v_0^2$ & $0\text{ or }1$ & $(1,1,2,2)$\\
$v_r^2$, $1\leq r\leq3$ & $0\text{ or }1$ & $(1,1,2,0)$\\
$v_r^m$, $3\leq m\leq k$, $0\leq r\leq3$
& $0\text{ or }1$ & $(1,1,2,0)$\\
\bottomrule
\end{tabular}
\end{table}

Rows of Table~\ref{tab:C6-root-response} corresponding to illegal moves are ignored. If $z$ has index $4$ or $5$, reflect every index in the corresponding row. If Alice's nonzero color was originally $2$, interchange the names of colors $1$ and $2$ in that row.

For each entry in Table~\ref{tab:C6-root-response}, Bob's displayed move is legal. The two target vertices are $v_{\varepsilon}^0$ and $v_{\varepsilon}^1$, and the third color is legal at the uncolored vertices $v_p^1$ and $v_q^2$. If Bob's move has already made a target uncolorable, he wins immediately; otherwise Lemma~\ref{lem:two-layer-fork} applies. The table covers all legal moves by Alice, so Bob also wins this exceptional case. This proves the lower bound $\chi_g(M_k(C_n))\geq4$ for every $n\geq5$.

For the upper bound, apply the cycle case of Theorem~\ref{thm:uniform-five}: Alice colors $\omega$ first, after which every uncolored vertex has degree at most four and therefore has a legal color among the five available colors. \qedhere

\end{proof}

\section*{Conflicts of interest}
The authors declare no conflict of interest.

\section*{Acknowledgments}

This work was supported by the National Natural Science Foundation of China (Grant No.~12461006), and the Guizhou Provincial Basic Research Program (Grant Nos.~QN[2025]020, ZD[2025]085, MS[2026]141, KJLYRC[2026]091).


\begin{thebibliography}{99}
\expandafter\ifx\csname url\endcsname\relax
  \def\url#1{\texttt{#1}}\fi
\expandafter\ifx\csname urlprefix\endcsname\relax\def\urlprefix{URL }\fi
\expandafter\ifx\csname href\endcsname\relax
  \def\href#1#2{#2} \def\path#1{#1}\fi

\bibitem{Alagammai2019}
{R. Alagammai, V. Vijayalakshmi, Game chromatic number and game chromatic index of the Mycielski graphs of some families of graphs, Appl. Math. Inf. Sci. 13(S1) (2019) 261--265.}

\bibitem{Bartnicki2007}
{T. Bartnicki, J. Grytczuk, H.A. Kierstead, X. Zhu, The map-coloring game, Amer. Math. Monthly 114(9) (2007) 793--802.}

\bibitem{Bodlaender1991}
{H.L. Bodlaender, On the complexity of some coloring games, Int. J. Found. Comput. Sci. 2 (1991) 133--147.}

\bibitem{Chamberlin2019}
{C. Chamberlin, J. DeCapua, H. Elser, D. Gerraputa, A. Hamm, On game chromatic number analogues of Mycielsians and Brooks' Theorem, North Carolina J. Math. Stat. 5 (2019) 17--27.}

\bibitem{Chang1999}
{G.J. Chang, L. Huang, X. Zhu, Circular chromatic numbers of Mycielski's graphs, Discrete Math. 205 (1999) 23--37.}

\bibitem{Chen2014}
{M. Chen, X. Guo, H. Li, L. Zhang, Total chromatic number of generalized Mycielski graphs, Discrete Math. 334 (2014) 48--51.}

\bibitem{Costa2019}
{E.R. Costa, V.L. Pessoa, R. Sampaio, R. Soares, PSPACE-hardness of two graph coloring games, Electron. Notes Theor. Comput. Sci. 346 (2019) 333--344.}

\bibitem{Enomoto2023}
{H. Enomoto, J. Fujisawa, N. Matsumoto, Game chromatic number of strong product graphs, Discrete Math. 346(1) (2023) 113162.}

\bibitem{Faigle1993}
{U. Faigle, U. Kern, H.A. Kierstead, W.T. Trotter, On the game chromatic number of some classes of graphs, Ars Combin. 35 (1993) 143--150.}

\bibitem{Fan2004}
{G. Fan, Circular chromatic number and Mycielski graphs, Combinatorica 24(1) (2004) 127--135.}


\bibitem{Frieze2013}
{A. Frieze, S. Haber, M. Lavrov, On the game chromatic number of sparse random graphs, SIAM J. Discrete Math. 27(2) (2013) 768--790.}

\bibitem{GuanZhu1999}
{D. Guan, X. Zhu, Game chromatic number of outerplanar graphs, J. Graph Theory 30 (1999) 67--70.}

\bibitem{Hollom2025}
{L. Hollom, On graphs with maximum difference between game chromatic number and chromatic number, Discrete Math. 348(2) (2025) 114271.}

\bibitem{Lam2003}
{P.C.B. Lam, W. Lin, G. Gu, Z. Song, Circular chromatic number and a generalization of the construction of Mycielski, J. Combin. Theory Ser. B 89(2) (2003)
195--205.}

\bibitem{Larsen1995}
{M. Larsen, J. Propp, D. Ullman, The fractional chromatic number of Mycielski's graphs, J. Graph Theory 19(3) (1995) 411--416.}

\bibitem{Laso2017}
{M. Laso\'{n}, The coloring game on matroids, Discrete Math. 340 (2017) 796--799.}

\bibitem{Lin2006}
{W. Lin, J. Wu, P.C.B. Lam, G. Gu, Several parameters of generalized Mycielskians, Discrete Appl. Math. 154(8) (2006) 1173--1182.}

\bibitem{Liu2004}
{D.F. Liu, Circular chromatic number for iterated Mycielski graphs, Discrete Math. 285 (2004) 335--340.}

\bibitem{Miskuf2009}
{J. Mi\v{s}kuf, R. \v{S}krekovski, M. Tancer, Backbone colorings and generalized Mycielski graphs, SIAM J. Discrete Math. 23(2) (2009) 1063--1070.}

\bibitem{Mycielski1955}
{J. Mycielski, Sur le coloriage des graphes, Colloq. Math. 3 (1955) 161--162.}


\bibitem{Zhuproduct2008}
{X. Zhu, Game coloring the Cartesian product of graphs, J. Graph Theory 59 (2008) 261--278.}

\bibitem{Zhu2008}
{X. Zhu, Refined activation strategy for the marking game, J. Combin. Theory Ser. B 98(1) (2008) 1--18.}









\end{thebibliography}
\end{document}